\documentclass[12pt]{amsart}
\usepackage[colorlinks=true,pagebackref,hyperindex,citecolor=green,linkcolor=red]{hyperref}
\usepackage{fullpage}
\usepackage{setspace}
\usepackage{amsmath}
\usepackage{amsfonts}
\usepackage{amssymb}
\usepackage{bm}
\usepackage{mathrsfs}
\usepackage{mathabx}
\usepackage{color}
\usepackage[all]{xy}
\usepackage{graphicx}
\usepackage{booktabs}
\usepackage{pstricks}

\usepackage{mdwlist} %for suspending enumeration

 \newtheorem{thm}{Theorem}[section]
 \newtheorem{Corollary}[thm]{Corollary}
 \newtheorem{Lemma}[thm]{Lemma}
 \newtheorem{Proposition}[thm]{Proposition}

 \newtheorem*{TheoremIntro}{Theorem}
 \newtheorem*{mainTheorem}{Theorem \ref{MainTheorem}}
 \newtheorem{KeyLemma}[thm]{Key Lemma}

\newtheorem{UHypothesis}[thm]{{Universal Hypothesis}}

\newtheorem*{fptsetPropintro}{Proposition \ref{fptInterval: P}}

 \theoremstyle{definition}
 \newtheorem{Definition}[thm]{Definition}
 \newtheorem{Remark}[thm]{Remark}

 \newtheorem{Example}[thm]{Example}

 \numberwithin{equation}{subsection}

\newcommand{\up}[1]{\left\lceil #1 \right\rceil}
\newcommand{\down}[1]{\left\lfloor #1 \right\rfloor}

\newcommand{\mfrak}[1]{\mathfrak{#1}}

\newcommand{\Hom}{\operatorname{Hom}}

\renewcommand{\tilde}{\widetilde}

\newcommand{\m}{\mfrak{m}}

\renewcommand{\bold}[1]{\mathchoice{\hbox{\boldmath $\displaystyle #1$}}
        {\hbox{\boldmath $\textstyle #1$}}
        {\hbox{\boldmath $\scriptstyle #1$}}
        {\hbox{\boldmath $\scriptscriptstyle #1$}}}

\newcommand{\fpt}[1]{\bold{\operatorname{fpt}}(#1)}
\newcommand{\tr}[2]{\left \langle {#1} \right \rangle_{#2}} 
\newcommand{\tail}[2]{ \left \ldbrack {#1} \right \rdbrack_{#2}}
\renewcommand{\repeat}[2]{ \overline{\tr{#1}{}}_{#2}}

\newcommand{\lct}[1]{\bold{\operatorname{lct}}(#1)}

\newcommand{\fptset}[1]{\bold{\operatorname{FPT}}_{#1}}

\renewcommand{\hom}[1]{\Hom_R(R^{1/p^{#1}}, R)}
\newcommand{\R}[1]{R^{1/p^{#1}}}

\renewcommand{\(}{\left(}
\renewcommand{\)}{\right)}
\newcommand{\rref}[1]{(\ref{#1})}
\newcommand{\set}[1]{ \left\{ \, #1 \, \right\}}
\newcommand{\pair}[3]{\left( {#1}, {#3} \bullet {#2} \right)}

\newcommand{\error}{\varepsilon}

\renewcommand{\th}{\text{th}}

\newcommand{\dee}{d}

\newcommand{\ux}{\underline{\bold{x}}}

\newcommand{\bracket}[2]{ {#1}^{[p^{#2}]}}

\begin{document}
% \onehalfspace

\title{$F$-purity of hypersurfaces}
\author{ Daniel J. Hern\'andez }
\thanks{The author was partially supported by the National Science Foundation RTG grant number 0502170 at the University of Michigan.}

\begin{abstract}  Motivated by connections with birational geometry over $\mathbb{C}$, the theory of $F$-purity for rings of positive characteristic may be extended to a theory of $F$-purity for ``pairs'' \cite{HW2002}.  Given an element $f$ of an $F$-pure ring of positive characteristic, this extension allows us to define the $F$-pure threshold of $f$, denoted $\fpt{f}$.  This invariant measures the singularities of $f$, and may be thought of as a positive characteristic analog of the log canonical threshold, an invariant that typically appears in the study of singularities of hypersurfaces over $\mathbb{C}$.  In this note, we study $F$-purity of pairs, and show (as is the case with log canonicity over $\mathbb{C}$) that $F$-purity is preserved at the $F$-pure threshold.  We also characterize when $F$-purity is equivalent to \emph{sharp} $F$-purity, an alternate notion of purity for pairs introduced in \cite{Schwede2008}.  These results on purity at the threshold generalize results appearing in \cite{Hara2006,Schwede2008}, and were expected to hold by many experts in the field.  We conclude by extending results in \cite{BMS2009} on the set of all $F$-pure thresholds to the most general setting. 
\end{abstract}

\maketitle
%\tableofcontents

\section*{Introduction}

Let $R$ be a domain of characteristic $p>0$.  The \emph{$e^{\th}$-iterated Frobenius} map $R \stackrel{F^e}{\to} R$ (defined by $r \mapsto r^{p^e}$) is a ring homomorphism whose image is the subring $R^{p^e} \subseteq R$ consisting of all $\( p^e \)^{\th}$ powers of elements of $R$.  The Frobenius map has been an important tool in commutative algebra since Kunz characterized regular rings  as those for which $R$ is flat over $R^{p^e}$ \cite{Kunz}.  In general, singular rings exhibit pathological behavior with respect to Frobenius, and by imposing conditions on the structure of $R$ as an $R^{p^e}$-module, new classes of singularities can be defined.  For example, we say that $R$ is \emph{$F$-finite} if $R$ is a finitely generated $R^{p^e}$ module for every (equivalently, for some) $e \geq 1$. We call $R$ \emph{$F$-pure} (or \emph{$F$-split})  if the inclusion $R^{p^e} \subseteq R$ splits as a map of $R^{p^e}$-modules for all (equivalently, for some) $e \geq 1$. \cite{HR1976}.  The notion of $F$-purity is a critical ingredient in the proof of the well-known Hochster-Roberts Theorem on the Cohen-Macaulay property of rings of invariants \cite{HR1974}.

%Recently, more subtle applications of Frobenius to define singularities has led to new classes of \emph{$F$-singularities} (e.g.\ \emph{$F$-regularity} and \emph{$F$-rationality}).    These singularities are  closely related  to the singularity types of varieties defined over $\mathbb{C}$ (e.g. \emph{log canonical} and \emph{rational singularities}) that appear in the so-called Minimal Model Program \cite{Fed1983, Karen1997b, Hara1998, Karen2000, KM1998}.  Motivated by this connection to the Minimal Model Program, we shift our perspective and consider \emph{singularities of pairs}.  

By modifying the condition that $R^{p^e} \subseteq R$ splits over $R^{p^e}$, one may extend the notion of purity for rings to a more general setting.  A \emph{pair}, denoted $\pair{R}{f}{\lambda}$, consists of the combined information of an ambient ring $R$, a non-zero, non-unit element $f \in R$, and a non-negative real parameter $\lambda$.  We say that the pair $\pair{R}{f}{\lambda}$ is \emph{$F$-pure} if the inclusion $R^{p^e} \cdot f^{\down{ (p^e-1) \lambda}} \subseteq R$ splits as a map of $R^{p^e}$-modules for all $e \gg 0$ \cite{HW2002}.  Here, $R^{p^e} \cdot f^N$ denotes the $R^{p^e}$-submodule of $R$ generated by $f^N$. The purity condition for pairs encapsulates the classical one, as the ring $R$ is $F$-pure if and only if the pair $\pair{R}{f}{0}$ is $F$-pure.  

Though technical, this extension adds great flexibility to the theory, and allows one to define \emph{$F$-pure thresholds}.  The $F$-pure threshold of $f$, denoted $\fpt{f}$,  is the supremum over all $\lambda \geq 0$ such that the pair $\pair{R}{f}{\lambda}$ is $F$-pure.  This definition is analogous to that of the \emph{log canonical threshold} in complex algebraic geometry, which we now briefly recall. 

Let $S$ be a regular ring of finite type over $\mathbb{C}$, and let $g$ be any non-zero, non-unit element of $S$.  Via (log) resolution of singularities, one may define the notion of \emph{log canonical singularities} for pairs $\pair{S}{g}{\lambda}$ \cite{Laz2004, BL2004}.  When the ambient space is a polynomial ring over $\mathbb{C}$, we have the following concrete description: $\pair{\mathbb{C}[\ux]}{g}{\lambda}$ is said to be \emph{log canonical} if  for every $0 \leq \error < \lambda$, the real-valued function $\frac{1}{|g|^{2 \error}}$ is locally integrable in a neighborhood of every point.  We define the \emph{log canonical threshold} of $g$, denoted $\lct{g}$, to be the supremum over all $\lambda \geq 0$ such that $\pair{S}{g}{\lambda}$ is log canonical.  The following theorem, one of many relating singularities defined over $\mathbb{C}$ with those in positive characteristic, illustrates the tight relationship between $F$-purity and log canonical singularities.  

\begin{TheoremIntro} \cite{HW2002}
%\label{HW: T}
Let $g \in S$ be as above, and let $g_p \in S_p$ denote the reduction of $g$ and $S$ to prime characteristic $p > 0$; see \cite{Karen1997} for a concrete discussion of this process.  If $\pair{S_p}{g_p}{\lambda}$ is $F$-pure for infinitely many $p$, then $\pair{S}{f}{\lambda}$ is log canonical.  
\end{TheoremIntro}

\noindent The converse is conjectured to hold; see \cite{Polynomials, Takagi2011} for recent positive results.  

In \cite{Schwede2008}, we are introduced to an alternate notion of purity for pairs called \emph{sharp $F$-purity}.  A pair $\pair{R}{f}{\lambda}$ is said to be sharply $F$-pure if $R^{p^e} \cdot f^{\up{(p^e-1) \lambda}} \subseteq R$ splits for some $e \geq 1$, and the above-mentioned theorem holds after replacing ``$F$-pure'' with ``sharply $F$-pure'' \cite{Schwede2008}.  Given the close ties between (sharp) $F$-purity and log canonical singularities, one is motivated to ask whether certain properties of log canonicity also hold for (sharp) $F$-purity in the positive characteristic setting.

One such property of log canonical singularities, which follows essentially from its definition, is that $\pair{S}{g}{\lct{g}}$ must be log canonical.  We show that the analogous property, though not an immediate consequence of the definitions involved, also holds for $F$-purity.  The situation for sharp $F$-purity is slightly more complicated.

\begin{mainTheorem} If $R$ is an $F$-pure ring of characteristic $p>0$, then $\pair{R}{f}{\fpt{f}}$ is $F$-pure, and is sharply $F$-pure  if and only if $(p^e-1) \cdot \fpt{f} \in \mathbb{N}$ for some $e \geq 1$.
\end{mainTheorem}

The condition that $(p^e-1) \cdot \fpt{f} \in \mathbb{N}$ appearing above is equivalent to the condition that $\fpt{f}$ be a rational number whose denominator not divisible by $p$.  While it is known that $\fpt{f} \in \mathbb{Q}$ in many cases (see  \cite{BMS2008, BMS2009, KLZ2009, BSTZ2009}), explicit computations of $F$-pure thresholds appearing in the literature show that very often the denominator of $\fpt{f}$ is divisible by $p$ \cite{Diagonals, Binomials}.  Thus, there are many instances in which $F$-purity and sharp $F$-purity are not equivalent.  However, from the point of view of computations of $F$-pure thresholds of hypersurfaces (especially those reduced from characteristic zero), the condition that $(p^e-1) \cdot \fpt{f} \in \mathbb{N}$ has many  desirable consequences.  We emphasize that Theorem \ref{MainTheorem} holds assuming only that the ambient ring is $F$-pure, which is the minimal assumption needed to study $F$-purity of pairs.  We also note that Theorem \ref{MainTheorem} generalizes results appearing in \cite{Hara2006,Schwede2008}, in which the ambient ring is assumed to be an $F$-finite, (complete) regular local ring.

If $g$ again denotes an element of a regular ring of finite type over $\mathbb{C}$, it follows easily from the definition of log canonicity (in terms of resolution of singularities) that $\lct{g} \in \mathbb{Q} \cap [0,1]$.  Furthermore, every number in $\mathbb{Q} \cap [0,1]$ may be realized as a log canonical threshold:  if $\frac{r}{s} \in \mathbb{Q} \cap [0,1]$, then $\frac{r}{s} = \lct{x_1^s + \cdots + x_r^s}$ \cite[Example 3]{Howald2001a}, \cite[Example 9]{Howald2001b}.  The issue of which numbers may be realized as $F$-pure thresholds is more complicated, and was first considered in \cite{BMS2009}.  

As with log canonical thresholds, it is easy to verify (e.g., see Lemma \ref{BasicProperties: L})  that the $F$-pure threshold of a hypersurface is contained in $[0,1]$.  If one considers only ambient rings that are $F$-finite and regular of a fixed characteristic $p>0$, it was shown in  \cite[Proposition 4.3]{BMS2009} that there exist infinitely many non-empty open intervals contained in $\mathbb{Q} \cap [0,1]$ that cannot contain a number of the form $\fpt{f}$;  the arguments given therein rely on the behavior of \emph{$F$-jumping exponents} in $F$-finite regular rings.  In Proposition \ref{fptInterval: P}, we show that the aforementioned statement still holds if one replaces the condition that the ambient spaces be $F$-finite and regular with the minimal requirement that they be $F$-pure.

\begin{fptsetPropintro}
Let $\fptset{p}$ denote the set of all $\fpt{f}$, where $f \in R$ ranges over every element of every $F$-pure ring of characteristic $p>0$; see Definition \ref{fptsetDefinition}.  Then, for every $e \geq 1$ and every $\beta \in [0,1] \cap \frac{1}{p^e} \cdot \mathbb{N}$, we have that \[ \fptset{p}  \cap \( \hspace{.05in} \beta, \frac{p^e}{p^e-1} \cdot \beta \hspace{.05in} \) = \emptyset. \]
\end{fptsetPropintro}

\begin{Example} 
\label{fptInterval: E}
Proposition \ref{fptInterval: P} states that, for every $e \geq 1$, there exist $p^e-1$ disjoint open subintervals of $[0,1]$ that do not intersect $\fptset{p}$.  To better appreciate this condition, consider the intervals corresponding to $e=1,2$ and $3$ that cannot intersect $\fptset{2}$.

\begin{center}
\psset{unit=5.5in}

\begin{pspicture}(0,-.00001)(1,.00001)
\psset{linewidth=1pt,linestyle=dotted,dotsep=1pt}
\psline[arrows=|-](0,0)(.5,0)
\psset{linestyle=solid, arrows=o-o,linewidth=1pt,linecolor=blue}
\psline(.5,0)(1,0)
\put(-.08,-.01){$e=1$}
\end{pspicture}

\begin{pspicture}(0,-.00001)(1,.00001)
\psset{linewidth=1pt,linestyle=dotted,dotsep=1pt}
\psline[arrows=|-](0,0)(.25,0)
\psline(.3333,0)(.5,0)
\psline(.666,0)(.75,0)
\psset{linestyle=solid, dotstyle=o,arrows=o-o,linewidth=1pt, linecolor=blue}
\psline(.25,0)(.3333,0)
\psline(.5,0)(.6666,0)
\psline(.75,0)(1,0)
\put(-.08,-.01){$e=2$}
\end{pspicture}

\begin{pspicture}(0,-.00001)(1,.00001)
\psset{linewidth=1pt,linestyle=dotted,dotsep=1pt}
\psline[arrows=|-](0,0)(.125,0)
\psline(.142857,0)(.25,0)
\psline(.285714,0)(.375,0)
\psline(.428571,0)(.5,0)
\psline(.571428,0)(.625,0)
\psline(.714285,0)(.75,0)
\psline(.857142,0)(.875,0)
\psset{linestyle=solid,arrows=o-o,linewidth=1pt,linecolor=blue}
\psline(.125,0)(.142857,0)
\psline(.25,0)(.285714,0)
\psline(.375,0)(.428571,0)
\psline(.5,0)(.571428,0)
\psline(.625,0)(.714285,0)
\psline(.75,0)(.857142,0)
\psline(.875,0)(1,0)
\put(-.08,-.01){$e=3$}
\end{pspicture}

\begin{pspicture}(0,-.00001)(1,.00001)
\psline[linewidth=1pt,linestyle=solid, arrows=|-](0,0)(1,0)
\rput(0,-.04){$0$}
\psline[arrows=|-|](0,0)(.125,0)
\rput(.125,-.04){$\frac{1}{8}$}
\psline[arrows=-|](.125,0)(.25,0)
\rput(.25,-.04){$\frac{1}{4}$}
\psline[arrows=-|](.25,0)(.375,0)
\rput(.375,-.04){$\frac{3}{8}$}
\psline[arrows=-|](.375,0)(.5,0)
\rput(.5,-.04){$\frac{1}{2}$}
\psline[arrows=-|](.5,0)(.625,0)
\rput(.625,-.04){$\frac{5}{8}$}
\psline[arrows=-|](.625,0)(.75,0)
\rput(.75,-.04){$\frac{3}{4}$}
\psline[arrows=-|](.75,0)(.875,0)
\rput(.875,-.04){$\frac{7}{8}$}
\psline[arrows=-|](.875,0)(1,0)
\rput(1,-.04){$1$}
\end{pspicture} \\
\vspace{.4in} 
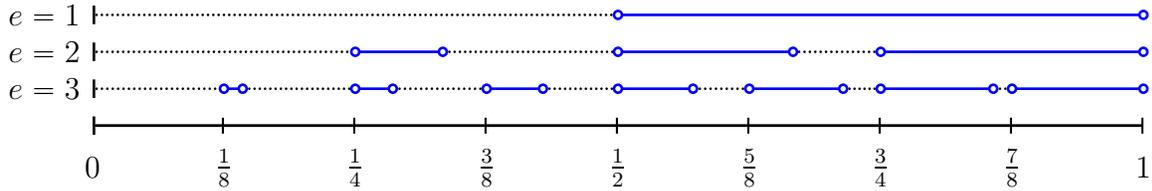
\textsc{Figure:}  Intervals corresponding to $e=1,e=2$, and $e=3$ that do not intersect $\fptset{2}$.
\end{center}
\end{Example}

As was first observed in \cite{BMS2009}, for every $e \geq 1$ we have that 
\[  \sum_{\beta \in [0,1] \cap \frac{1}{p^e} \cdot \mathbb{N}} \operatorname{length} \( \beta, \frac{p^e}{p^e-1} \cdot \beta \) = \frac{1}{2} .\] Thus, Proposition \ref{fptInterval: P} states that for every $e \geq 1$, there is a set of Lebesgue measure $\frac{1}{2}$ that does not intersect $\fptset{p}$.   Furthermore, Proposition \ref{fptInterval: P} may be used to show that $\fptset{p}$ is a set of Lebesgue measure zero.  This is not surprising because, as noted earlier, $F$-pure thresholds are very often rational numbers.  However, we stress that the issue of whether $\fptset{p} \subseteq \mathbb{Q}$ remains open in general.

The main results appearing in this article are obtained as corollaries of Key Lemma \ref{FTTruncationLemma}, which essentially says that all of the relevant ``splitting data'' for an element $f$ in an $F$-pure ring is encoded by the digits appearing in the unique (non-terminating) base $p$ expansion of $\fpt{f}$; see Definition \ref{ExpansionDefinition} for the definition of a non-terminating base $p$ expansion.  The proof of Key Lemma \ref{FTTruncationLemma} depends  mostly on taking $p^{\th}$ roots of elements and morphisms (see Definition \ref{rootofamap: D}) in a careful way, and in the case that the ambient space is $F$-finite and regular, Key Lemma \ref{FTTruncationLemma} is simply a translation of \cite[Proposition 1.9]{MTW2005} into the language of non-terminating base $p$ expansions.

%\subsection*{Outline}  In Section \ref{Statement: S}, we recall some basic facts and definitions from positive characteristic commutative algebra, and introduce the various types of $F$-singularities.  We conclude this section with a statement of our main result, Theorem \ref{MainTheorem}.  In Section \ref{KeyLemma: S}, we state our main tool, Key Lemma \ref{FTTruncationLemma}.  In Section \ref{FPTSet: S}, we study the properties of $F$-pure thresholds, and use these to derive strong conditions on the set of all $F$-pure thresholds in a fixed characteristic.  In Section \ref{Proof: S}, we prove our Main Theorem.  In Section \ref{Lemmas: S}, we prove the technical lemmas that underly our results, and in Section \ref{ProofKL: S}, we prove Key Lemma \ref{FTTruncationLemma}.  We conclude Section \ref{ProofKL: S} with some further corollaries of our Key Lemma.

\subsection*{Acknowledgements}  This note is part of the author's Ph.D thesis, which was completed at the University of Michigan under the direction of Karen Smith.  The author would like to thank Karen Smith, Mel Hochster, Emily Witt, Luis Nu$\tilde{\text{n}}$ez Betancourt, Daisuke Hirose, and Michael Von Korff for many enlightening discussions.

\section{Preliminaries}

Let $R$ be a reduced ring of prime characteristic $p>0$, and for $e \geq 1$, let $\R{e}$ denote the set of formal symbols $\{f^{1/p^e}:f \in R\}$.  We define a ring structure on $\R{e}$ via $f^{1/p^e} + g^{1/p^e} := (f+g)^{1/p^e}$ and $f^{1/p^e}~\cdot~g^{1/p^e}~:= (fg)^{1/p^e}$.  As $R$ is reduced, we have an inclusion $R \subseteq \R{e}$ given by $r \mapsto (r^{p^e})^{1/p^e}$.  If $R$ is a domain, then $\R{e}$ admits a more concrete description:   Let $L$ denote the algebraic closure of the fraction field of $R$, and for every $f \in R$, let $f^{1/p^e}$ denote the unique root of the equation $T^{p^e} - f \in L[T]$ in $L$.  We may then describe $\R{e}$ as the subring of $L$ consisting of all $\(p^e\)^{\th}$-roots of elements of $R$.  For example, if $R = \mathbb{F}_p[x_1, \cdots, x_m]$, then $\R{e} = \mathbb{F}_p[x_1^{1/p^e}, \cdots, x_m^{1/p^e}]$.  Via the inclusion (of rings) $\R{d} \subseteq \R{e+d}$ given by $r^{1/p^{d}} = (r^{p^e})^{1/p^{e + d}}$, we may identify $(\R{d})^{1/p^e}$ and $\R{e +d}$ as $\R{d}$-algebras. 

Let $R \stackrel{F^e}{\to} R$ denote the \emph{$e^{\th}$ iterated Frobenius morphism} defined by $r \mapsto r^{p^e}$, and let $F^e_{\ast} R$ denote $R$ when considered as an $R$-algebra via $F^e$.   If $R^{p^e} = \set{ r^{p^e} : r \in R}$ is the subring consisting of $\(p^e\)^{\th}$ powers of $R$, then the $R$-algebra structure of $F^e_{\ast} R$ and the $R^{p^e}$-algebra structure of $R$ are isomorphic.   We also note that $F^e_{\ast} R \cong \R{e}$ as $R$-modules via the isomorphism $r \mapsto r^{1/p^e}$.  This map and its inverse are often referred to as ``taking $(p^e)^{\th}$ roots'' and ``raising to $(p^e)^{\th}$ powers.''  We say that $R$ is \emph{$F$-finite} if $R^{1/p}$ (equivalently, $F_{\ast} R$) is a finitely-generated $R$-module.  One can show that $R$ is $F$-finite if and only if $\R{e}$ (equivalently, $F^e_{\ast}R$) is a finitely-generated $R$-module for some (equivalently, for all) $e \geq 1$.

If $S \subset T$ is an inclusion of rings, and $S \cdot t$ is the $S$-submodule of $T$ generated by $t \in T$, we say that the inclusion $S \cdot t \subseteq T$ \emph{splits over $S$} (or splits as a map of $S$-modules) if there exists a map $\theta \in \Hom_S \(T, S \)$ with $\theta \( t \) = 1$.  Recall that $R$ is said to be \emph{$F$-pure} (or \emph{$F$-split})  if the inclusion $R = R \cdot 1 \subseteq \R{}$ splits over $R$.  An $F$-pure ring is necessarily reduced, and $R$ is $F$-pure if and only if the inclusion $R \subseteq \R{e}$ splits as a map of $R$-modules for some (equivalently, for all) $e \geq 1$ \cite{HR1976}.  For $g \in R$, note that the inclusion $R \cdot g^{1/p^e} \subseteq R^{1/p^e}$ splits over $R$ if and only if $R^{p^e} \cdot g\subseteq R$ splits over $R^{p^e}$ if and only if $R \cdot g \subseteq F^e_{\ast} R$ splits over $R$.  As a matter of taste, we use the language of $(p^e)^{\th}$ roots when discussing such splittings.

\subsection{$F$-purity for pairs and $F$-pure thresholds}

\begin{Definition}
A \emph{pair} $\pair{R}{f}{\lambda}$ consists of the combined information of an ambient ring $R$, a hypersurface $f \in R$, and a non-negative real number $\lambda$.
\end{Definition}

\begin{UHypothesis}  For the remainder of this article, $R$ will be assumed to be an $F$-pure ring of characteristic $p>0$, and $f$ will be assumed to be a non-zero, non-unit in $R$.
\end{UHypothesis}

In what follows, $\up{\alpha}$ (respectively, $\down{\alpha}$) will denote the least integer greater than or equal to $\alpha$ (respectively, the greatest integer less than or equal to $\alpha$.)

\begin{Definition}
\label{PurityDefinition} \
 \cite{Tak2004, TW2004, Schwede2008}
The pair $\pair{R}{f}{\lambda}$ is said to be 
\begin{enumerate} 
\item \emph{$F$-pure} if $R \cdot f^{\down{(p^e-1) \lambda}/p^e} \subseteq \R{e}$ splits over $R$ for all $e \geq 1$,
\item \emph{strongly $F$-pure} if $R \cdot f^{\up{p^e \lambda}/p^e} \subseteq \R{e}$ splits over $R$ for some $e \geq 1$, and
\item \emph{sharply $F$-pure} if $R \cdot f^{\up{(p^e-1) \lambda}/p^e} \subseteq \R{e}$ splits over $R$ for some $e \geq 1$.
\end{enumerate}
\end{Definition}

Lemma \ref{SmallerExponent: L} shows that strong $F$-purity implies sharp $F$-purity. We also have that sharp $F$-purity 
implies  $F$-purity \cite{Schwede2008} (though the converse need not be true; see Example \ref{CounterExample: E}). 

\begin{Lemma}
\label{SmallerExponent: L}
If the inclusion $R \cdot f^{N/p^e} \subseteq \R{e}$ splits as a map of $R$-modules, so must the inclusion $R \cdot f^{a/p^e} \subseteq \R{e}$ for all $0 \leq a \leq N$.
\end{Lemma}

\begin{proof}  Choose $\theta \in \hom{e}$ with $\theta \( f^{N/p^e} \) = 1$, and let $\phi$ denote the $R$-linear endomorphism of $\R{e}$ given by multiplication by $f^{ \frac{N-a}{p^e}}$. Then $\theta \circ \phi \in \hom{e}$ and $\(\theta \circ \phi\)\(f^{a/p^e}\) = 1$.
\end{proof}

\begin{Lemma}
\label{BasicProperties: L} If $\pair{R}{f}{\lambda}$ is $F$-pure, then so is $\pair{R}{f}{\error}$ for every $0 \leq \error \leq \lambda$.  Furthermore, $\pair{R}{f}{\lambda}$ is not $F$-pure if $\lambda > 1$.
\end{Lemma}

\begin{proof} The pair $\pair{R}{f}{0}$ is $F$-pure as $R$ is $F$-pure, and applying  Lemma \ref{SmallerExponent: L} to the inequality $\down{\(p^e-1\) \error} \leq \down{ (p^e-1) \lambda}$ shows that $\pair{R}{f}{\error}$ is $F$-pure whenever $\pair{R}{f}{\lambda}$ is $F$-pure.  For the second assertion, suppose that $\pair{R}{f}{\lambda}$ is $F$-pure with $\lambda = 1 + \error$ for some $\error > 0$.  By definition, we have that \begin{equation} \label{<1: e} R \cdot f^{\down{ \( p^e-1 \) (1 + \error)}/p^e} \subseteq \R{e} \text{ splits over $R$ for every  } e \geq 1. 
\end{equation}    

For $e \gg 0, (p^e-1) (1 + \error) = p^e-1 + ( p^e-1 ) \cdot \error > p^e$ for $e \gg 0$.  Applying Lemma \ref{SmallerExponent: L} to \rref{<1: e}  then shows that $R \cdot f^{p^e/p^e} = R \cdot f \subseteq \R{e}$ splits over $R$ for $e \gg 0$.   We conclude that there exists a map $\theta \in \hom{e}$ with $1 = \theta(f) = f \cdot \theta(1)$, contradicting the assumption that $f$ is not a unit.
\end{proof}

Of course, one can replace $F$-purity with strong (respectively, sharp) $F$-purity in Lemma \ref{BasicProperties: L} to obtain analogous statements.  Lemma \ref{BasicProperties: L} shows that $F$-purity for a given a parameter implies $F$-purity for all smaller parameters, and so one may ask:   What is the largest parameter for which a pair is $F$-pure?  This leads to the notion of $F$-pure thresholds, which were first defined in \cite{TW2004}.  It was shown in \cite[Proposition 2.2]{TW2004} (respectively, \cite[Proposition 5.3]{Schwede2008}) that the $F$-pure threshold agrees with the ``strongly $F$-pure threshold''  (respectively,  ``sharply $F$-pure threshold'').  We gather these facts in Definition \ref{fpt: D2}.

\begin{Definition} 
\label{fpt: D2} The following supremums all agree, and we call their common value the \emph{$F$-pure threshold of $f$}, denoted $\fpt{R,f}$:
\begin{align*} \fpt{R,f} := \sup_{\lambda} \set{ \pair{R}{f}{\lambda} \text{ is $F$-pure} }  = &  \sup_{\lambda} \set{ \pair{R}{f}{\lambda} \text{ is strongly $F$-pure} } \\ = &  \sup_{\lambda} \set{ \pair{R}{f}{\lambda} \text{ is sharply $F$-pure} }. 
\end{align*}
We often suppress the ambient ring and write $\fpt{f}$ instead of $\fpt{R,f}$.
\end{Definition}

As a corollary of Lemma \ref{BasicProperties: L}, we see that $\fpt{R,f} \in [0,1]$.  
\begin{Remark}
\label{FT>0Remark} Note that $\fpt{f} > 0$ if and only if there exists $e \geq 1$ such that $R \cdot f^{1/p^e} \subseteq \R{e}$ splits over $R$.  One recognizes that $\fpt{f} > 0$ if $R$ is a \emph{strongly $F$-regular} domain \cite{HH1990}.
\end{Remark}

\begin{Remark}
The issue of whether one can replace ``$\sup$'' with $``\max$'' in Definition \ref{fpt: D2} is precisely the content of Theorem  \ref{MainTheorem}.
\end{Remark}

%\subsection{The Main Theorem}

%\begin{thm}[cf.\cite{Hara2006,Schwede2008}]
%\label{MainTheorem}The pair $\pair{R}{f}{\fpt{f}}$ is
%\begin{enumerate}
%\item $F$-pure,
%\item not strongly $F$-pure, and 
%\item sharply $F$-pure $\iff (p^e-1) \cdot \fpt{f} \in \mathbb{N}$ for some $e \geq 1$.
%\end{enumerate}
%\end{thm}

%\begin{Remark} The first point of Theorem \ref{MainTheorem} generalizes \cite[Proposition 2.6]{Hara2006}, in which it is assumed that $R$ is a  complete, regular local ring, and $\R{}$ is a finite $R$-module. The last point in Theorem \ref{MainTheorem} is a generalization of \cite[Corollary 5.4 and Remark 5.5]{Schwede2008}, in which it is assumed that $R$ is a regular local ring, and  $\R{}$ is a finite $R$-module.  
%\end{Remark}

%\begin{Remark}
%The condition that $(p^e-1) \cdot \fpt{f} \in \mathbb{N}$ for some $e \geq 1$ is equivalent to the condition that  $\fpt{f} \in \mathbb{Q}$, and that the denominator of $\fpt{f}$ is not divisible by $p$.
%\end{Remark}

%\begin{Remark}  For the pair in Example \ref{CounterExample: E}, one calculates that $\fpt{x^p} = \frac{1}{p}$.  This behavior appears to be typical:  The denominator of $\fpt{f}$ is often a \emph{power} of $p$, and more often is divisible by $p$ \cite{Diagonals,Binomials}.  Thus, there are many examples for which $F$-purity is not equal to sharp $F$-purity.

\section{Some remarks on base $p$ expansions}

 In this section, we will consider the base $p$ expansions of real numbers contained in the unit interval.  By Lemma \ref{BasicProperties: L},  $\fpt{f}$ is such a number, and in the next section we will use the language developed here to obtain properties of $f$ via $\fpt{f}$.

\begin{Definition}
\label{ExpansionDefinition}  If $\alpha \in (0,1]$, we call the (unique) expression $\alpha = \sum_{e \geq 1} \frac{a_e}{p^e}$ with the property that the integers $0 \leq a_e \leq p-1$ are not all eventually zero the \emph{non-terminating base $p$ expansion of $\alpha$}, and we call $a_e$ the \emph{$e^{\th}$ digit} of $\alpha$.  \end{Definition}

The adjective ``non-terminating'' in Definition \ref{ExpansionDefinition} is only necessary when $\alpha$ is a rational number with denominator a power of $p$.  For example, $\frac{0}{p} + \sum_{e \geq 2} \frac{p-1}{p^e}$ is the the non-terminating base $p$ expansion of $\frac{1}{p}$.

\begin{Definition}
\label{TruncationTailDefinition}

Consider $\alpha \in (0,1]$ with non-terminating base $p$ expansion $\alpha = \sum_{d \geq 1} \frac{a_d}{p^d}$. 

\begin{enumerate}  
\item We define the \emph{$e^{\th}$ truncation of $\alpha$} by $\tr{\alpha}{e}  :=  \frac{a_1}{p} + \cdots + \frac{a_e}{p^e}$.  By convention, $\tr{0}{e} = 0$.
\item We define the \emph{$e^{\th}$ tail of $\alpha$} by $\tail{\alpha}{e}:= \alpha - \tr{\alpha}{e} = \sum_{d \geq e+1} \frac{a_d}{p^d}$.  By convention, $\tail{0}{e} = 0$.
\end{enumerate}
\end{Definition}

\begin{UHypothesis} All truncations and tails will be taken with respect to some fixed prime $p$ (which will always be the characteristic of any ambient ring in context). 
\end{UHypothesis}

\begin{Lemma}
\label{TruncationTailLemma}  Let $\alpha \in (0,1]$.  Then the following hold:
\begin{enumerate}
\item $\tr{\alpha}{e} \in \frac{1}{p^e} \cdot \mathbb{N}$.
\item $\tr{\alpha}{e} < \alpha$ and $\tail{\alpha}{e} > 0$ for all $e$.
\item $\tail{\alpha}{e} \leq \frac{1}{p^e}$, with equality if and only if $\alpha \in \frac{1}{p^e} \cdot \mathbb{N}$.  
\end{enumerate}
\end{Lemma}

\begin{proof} These all follow easily from the definitions, keeping in mind that we are dealing with objects derived from non-terminating expansions.  The details are left to the reader.
\end{proof}
 
 \begin{Lemma}
\label{RoundingLemma}
Let $\alpha \in (0,1]$.  Then the following hold:
\begin{enumerate}
\item $\up{ p^e \alpha } = p^e \tr{\alpha}{e} + 1$.
\item $p^e \tr{\alpha}{e} -1 \leq \down{ (p^e-1) \alpha} \leq p^e \tr{\alpha}{e}$.
\item $p^e \tr{\alpha}{e} \leq \up{(p^e-1) \alpha}  \leq p^e \tr{\alpha}{e} + 1$.
\item If $\beta \in [0,1] \cap \frac{1}{p^e} \cdot \mathbb{N}$ and $\alpha > \beta$, then $\tr{\alpha}{e} \geq \beta$.
\end{enumerate}

\end{Lemma}

\begin{proof}
By definition, we have that \begin{equation} \label{RLeq1} p^e \alpha = p^e \tr{\alpha}{e} + p^e \tail{\alpha}{e}. \end{equation}  By Lemma \ref{TruncationTailLemma}, the first summand in \eqref{RLeq1} is an integer, while the second is contained in $(0,1]$. The first point follows.  Substituting the decomposition of $p^e \alpha$ appearing in \eqref{RLeq1} into the expression $(p^e -1 ) \cdot \alpha$ shows that  \begin{equation} \label{RLeq2} (p^e-1) \cdot \alpha = p^e \tr{\alpha}{e} + p^e \tail{\alpha}{e} - \alpha.\end{equation}  By Lemma \ref{TruncationTailLemma}, both $\alpha$ and $p^e \tail{\alpha}{e}$ are contained in $(0,1]$, so that $ | p^e \tail{\alpha}{e} - \alpha |  < 1$.  Thus, the second and third  points follow from \eqref{RLeq2}.  We now prove the last point.  By Lemma \ref{TruncationTailLemma}, we have that $\frac{1}{p^e} + \tr{\alpha}{e} \geq \alpha > \beta$, and multiplying by $p^e$ shows that \begin{equation} \label{rounding: e} 1 + p^e \tr{\alpha}{e} > p^e \beta.\end{equation}  By assumption, both sides of \ref{rounding: e} are integers, and we conclude that $p^e \tr{\alpha}{e} \geq p^e \beta$.
 \end{proof}

\begin{Lemma}
\label{RepeatingExpansion: L}  Let $\alpha \in (0,1]$.  If $(p^{d} - 1 ) \cdot \alpha \in \mathbb{N}$, then $\tr{\alpha}{e d+d} =  \tr{\alpha}{e d} + \frac{1}{p^{e d}} \cdot \tr{\alpha}{d}$  for all $e \geq 1$.
\end{Lemma}

\begin{proof}  One can show that if $(p^d - 1) \cdot \alpha \in \mathbb{N}$, then the digits of the non-terminating base $p$ expansion of $\alpha$ must begin repeating after the $d^{\th}$ digit.  Once this has been established, it is easy to see that the $(ed + d)^{\th}$ truncation of $\alpha$ is determined by the $ed^{\th}$ and $d^{\th}$ truncations in precisely the prescribed way.  The details are left to the reader.\end{proof}

\section{Proof of the Key Lemma}

The aim of this section is to prove Key Lemma \ref{FTTruncationLemma}, which says that all of the relevant splitting data for $f$ is encoded by the truncations of $\fpt{f}$.  The process of taking roots of maps is key to the proof of Key Lemma \ref{FTTruncationLemma}, so we isolate it in Definition \ref{rootofamap: D}.

\begin{Definition}
\label{rootofamap: D}
An $R$-linear map $\theta: \R{e} \to R$ gives rise, in a natural way, to an $\R{d}$-linear map $\theta^{1/p^{\dee}}: \R{e + \dee} \to \R{\dee}$ defined by the rule $\theta^{1/p^{\dee}} (r^{1/p^{e + \dee}}) := \theta(r^{1/p^e})^{1/p^{\dee}}$.  We call $\theta^{1/p^{\dee}}$ the \emph{$(p^{\dee})^{th}$-root of $\theta$}.
\end{Definition}

\begin{Remark}
\label{lastsecond: R}
As $R$ is always assumed to be $F$-pure, $R \subseteq \R{}$ splits over $R$.  By taking $(p^e)^{\th}$ roots of this splitting, we see that $\R{e} \subseteq \R{e+1}$ splits over $\R{e}$, and a slight modification of this argument shows that $\R{e} \subseteq \R{d}$ splits over $\R{e}$ for every $d > e$. 
\end{Remark}

\begin{Lemma}
\label{SplittingLemma}
Let $\alpha \in [0,1] \cap \frac{1}{p^{\dee}} \cdot \mathbb{N}$. If the inclusion $R \cdot f^{\up{p^e \alpha}/p^e} \subseteq \R{e}$ splits over $R$ for some $e \geq 1$, then so must the inclusion $R \cdot f^{\alpha} \subseteq \R{\dee}$.
\end{Lemma}

\begin{proof} Fix $\theta \in \hom{e}$ with the property that  $\theta(f^{\up{p^e \alpha}/p^e}) =1$. If $e \geq \dee$, then $\alpha \in \frac{1}{p^{\dee}} \cdot \mathbb{N} \subseteq \frac{1}{p^e} \cdot \mathbb{N}$, and so $f^{\alpha} \in \R{\dee} \subseteq \R{e}$.  It is then clear that $f^{\up{p^e \alpha}/p^e} = f^{\alpha}$ maps to $1$ under the composition $\R{\dee} \subseteq \R{e} \stackrel{\theta}{\longrightarrow} R$.

Instead, suppose that $\dee > e$, so that $\R{e} \subseteq \R{\dee}$.    Note that $p^e \alpha \leq \up{p^e \alpha}$, so $\alpha \leq \up{p^e \alpha} / p^e$.  By Lemma \ref{SmallerExponent: L}, it suffices to show that there exists an $R$-linear map $\R{\dee} \to R$ sending $f^{\up{p^e \alpha}/p^e}$ to $1$.  By Remark \ref{lastsecond: R}, there exists an $\R{e}$-linear map $\phi: \R{\dee} \to \R{e}$ with $\phi(1)=1$.  Note that the $\R{e}$-linearity of $\phi$ implies that $\phi \( f^{\up{p^e \alpha}/p^e} \) = f^{\up{p^e \alpha}/p^e}$.  Thus, if  $\Theta$ denotes the composition $\R{\dee} \stackrel{\phi}{\to} \R{e} \stackrel{\theta}{\longrightarrow} R$, then $\Theta \( f^{ \up{p^e \alpha}/p^e} \) = 1$.
\end{proof}

\begin{KeyLemma}
\label{FTTruncationLemma} $p^{\dee} \tr{\fpt{f}}{\dee} = \max \set{ a \in \mathbb{N} : R \cdot f^{a/p^{\dee}} \subseteq \R{\dee} \text{ splits over $R$}}$. 
\end{KeyLemma}

\begin{proof} Set $\lambda:=\fpt{f}$ and $\nu_f(p^{\dee}):=\max \left \{ a \in \mathbb{N}: R \cdot f^{a/p^{\dee}} \subseteq \R{\dee} \text{ splits over $R$ } \right \}$.  As $R$ is $F$-pure and $f$ is a non-unit of $R$, we have that $0 \leq \nu_f(p^{\dee}) \leq p^{\dee}-1$.  If $\lambda = 0$, then $\nu_f(p^{\dee}) = 0$, and there is nothing to prove.  We will now assume that $\lambda \in (0,1]$.  By Lemma \ref{TruncationTailLemma},  $\tr{\lambda}{\dee} < \lambda$, so it follows from Definition \ref{fpt: D2} that $\pair{R}{f}{\tr{\lambda}{\dee}}$ is strongly $F$-pure, and hence there exists $e \geq 1$ such that $R \cdot f^{\up{p^e \tr{\lambda}{\dee}}/p^e} \subseteq \R{e}$ splits. By Lemma \ref{SplittingLemma}, we may assume that $e = \dee$, and it follows that $p^{\dee} \tr{\lambda}{\dee}~\leq~\nu_f(p^{\dee})$.  

We now show that the inclusion $R \cdot f^{\tr{\lambda}{\dee} +\frac{1}{p^{\dee}}} \subseteq \R{\dee}$ \emph{never} splits over $R$.  By Lemma \ref{RoundingLemma}, $p^{\dee} \tr{\lambda}{\dee} + 1 = \up{p^{\dee} \lambda}$,  so it suffices to show that $\theta\(f^{\up{p^{\dee}\lambda}/p^{\dee}} \) \neq 1$ for every $\theta \in \hom{\dee}$.  If $\lambda \notin \frac{1}{p^{\dee}} \cdot \mathbb{N}$, then $\up{p^{\dee} \lambda} = \up{p^{\dee} (\lambda + \varepsilon)}$ for $0 < \varepsilon \ll 1$.  By Definition \ref{fpt: D2}, $\lambda$ is also the ``strongly $F$-pure threshold,''  and it follows that $\theta\(f^{\up{p^{\dee} \lambda}/p^{\dee}}\) = \theta\(f^{\up{p^{\dee} (\lambda + \varepsilon)}/p^{\dee}} \) \neq 1$ for every $\theta \in \hom{\dee}$.

Now, suppose that $\lambda \in \frac{1}{p^{\dee}} \cdot \mathbb{N}$, so that $\up{p^{\dee} \lambda} / p^{\dee} = \lambda$.  By way of contradiction, suppose that $\theta \( f^{\lambda} \) = 1$ for some $\theta \in \hom{\dee}$.   Note that  $0 \neq \lambda \geq \frac{1}{p^{\dee}}$, so by Lemma \ref{SmallerExponent: L}, there exists an $R$-linear map $\R{\dee} \to R$ sending $f^{1/p^{\dee}}$ to $1$.  Taking $(p^{\dee})^{\th}$ roots of this map produces an $\R{\dee}$-linear map $\phi:\R{2\dee} \to \R{\dee}$ with the property that $\phi(f^{1/p^{2\dee}})=1$. Under the composition $\R{2 \dee} \stackrel{\phi}{\longrightarrow} \R{\dee} \stackrel{\theta}{\longrightarrow} R$, it follows from the $\R{\dee}$-linearity of $\phi$ that \[ f^{\lambda + \frac{1}{p^{2\dee}}} = f^{\lambda} \cdot f^{1/p^{2\dee}} \mapsto f^{\lambda} \cdot \phi(f^{1/p^{2\dee}}) = f^{\lambda} \cdot 1 \mapsto \theta( f^{\lambda}) = 1.\]  We see that $R \cdot f^{\lambda +\frac{1}{p^{2\dee}}} \subseteq \R{ 2\dee}$ splits, contradicting the definition of $\lambda =\fpt{f}$.
\end{proof}

 \begin{Remark}
\label{AlgorithmRemark}  Key Lemma \ref{FTTruncationLemma} allows us to give a definition of $\fpt{f}$ which makes no reference to $F$-purity of pairs, which we now briefly describe.  Set  \[ a_1 = \max \{ a \in \mathbb{N}:  R \cdot f^{a/p} \subseteq \R{} \text{ splits} \}.\]  As $R$ is $F$-pure and $f$ is a non-unit, $0 \leq a_1 \leq p-1$.  Inductively, define \[ a_{e+1} = \max \{ a \in \mathbb{N}: R \cdot f^{a_1/p + \cdots + a_e/p^e + a/p^{e+1}} \subseteq \R{e+1} \text{ splits  as a map of $R$-modules} \}.\]  Again, one may verify that $0 \leq a_e \leq p-1$.  Furthermore, by taking $p^{th}$ roots of maps as in the proof of Key Lemma \ref{FTTruncationLemma}, one may show that the $a_e$ are not all eventually zero as long as one $a_e \neq 0$ (i.e., $\fpt{f} \neq 0$).   In this case, Key Lemma \ref{FTTruncationLemma} tells us that $\fpt{f} = \sum_{e \geq 1} \frac{a_e}{p^e}$ is the non-terminating base $p$ expansion of $\fpt{f}$.
\end{Remark}

  \subsection{Some consequences of Key Lemma \ref{FTTruncationLemma}}
 
In this subsection, we gather some corollaries of Key Lemma \ref{FTTruncationLemma} that we will utilize in later sections.

\begin{Lemma}
\label{OneConditionSplittingLemma}
Let $\alpha \in [0,1]$ with $(p^{\dee}-1)\cdot \alpha \in \mathbb{N}$ . If $R \cdot f^{\tr{\alpha}{\dee}} \subseteq \R{\dee}$ splits over $R$, then so does $R \cdot f^{ \tr{\alpha}{e \dee}} \subseteq \R{e \dee}$ for every $e \geq 1$. 
\end{Lemma}

\begin{proof} We induce on $e$, the base case being our hypothesis. Suppose $R \cdot f^{\tr{\alpha}{\dee}} \subseteq \R{\dee}$ and $R \cdot f^{\tr{\alpha}{e \dee}} \subseteq \R{e \dee}$ split as maps of $R$-modules, so that there exist
\begin{enumerate}
\item \label{split1} an $R$-linear map $\R{\dee} \to R$ with $f^{\tr{\alpha}{\dee}} \mapsto 1$, and 
\item \label{split2} an $R$-linear map $\theta: \R{e \dee} \to R$ with $\theta(f^{\tr{\alpha}{e \dee}}) = 1$.
\suspend{enumerate}

\noindent We now show that $R \cdot f^{\tr{\alpha}{e \dee+\dee}} \subseteq \R{e \dee + \dee}$ splits as a map of $R$-modules.  By taking $(p^{e \dee})^{\th}$-roots of the map in \eqref{split1}, we obtain \
 
 \resume{enumerate} 
 \item \label{split 3} an $\R{e \dee}$-linear map $\phi: \R{e \dee + \dee} \to \R{e \dee}$ with $\phi \left(f^{ \frac{ \tr{\alpha}{\dee}}{p^{e \dee}}} \right)=1$.
  \end{enumerate}

\noindent By Lemma \ref{RepeatingExpansion: L}, we have that  $\tr{\alpha}{e \dee+\dee} =  \tr{\alpha}{e \dee} + \frac{\tr{\alpha}{\dee}}{p^{e \dee}}$  for all $e \geq 1$, and it follows that \begin{equation} \label{product} f^{\tr{\alpha}{e \dee+\dee}} =   f^{\frac{\tr{\alpha}{\dee}}{p^{e\dee}}} \cdot  f^{\tr{\alpha}{e\dee}}. \end{equation}  

\noindent Under the composition $\R{e \dee + \dee} \stackrel{\phi}{\longrightarrow} \R{e \dee} \stackrel{\theta}{\longrightarrow} R$, it follows from (\ref{product}) and the $\R{e \dee}$-linearity of $\phi$, that  $ f^{ \tr{\alpha}{e \dee + \dee} } = f^{\frac{\tr{\alpha}{\dee}}{p^{e\dee}}} \cdot  f^{\tr{\alpha}{e\dee}} \mapsto f^{\tr{\alpha}{e \dee}} \cdot \phi\left(f^{ \frac{ \tr{\alpha}{\dee}}{p^{e \dee}}} \right) = f^{\tr{\alpha}{e \dee}} \cdot 1 \mapsto \theta( f^{\tr{\alpha}{e \dee} }) = 1$. 
\end{proof}

\begin{Corollary}  
\label{OneConditionSplittingCorollary}
Let $\alpha \in [0,1]$ with $(p^{\dee}-1)\cdot \alpha \in \mathbb{N}$.  If $R \cdot f^{\tr{\alpha}{\dee}}  \subseteq \R{\dee}$ splits, then $\alpha \leq \fpt{f}$.
\end{Corollary}

\begin{proof}  By Lemma \ref{OneConditionSplittingLemma}, $R \cdot f^{\tr{\alpha}{e \dee}} \subseteq \R{e \dee}$ splits for all $e \geq 1$.  By Key Lemma \ref{FTTruncationLemma}, we have that $\tr{\alpha}{e \dee} \leq \tr{\fpt{f}}{e \dee}$ for every $e \geq 1$, and taking the limit as $e \to \infty$ gives the desired inequality. 
\end{proof}

%\begin{Corollary}  
%\label{OneConditionSplittingCorollary}
%Let $\alpha \in [0,1]$ with $(p^{\l}-1)\cdot \alpha \in \mathbb{N}$.  If $R \cdot f^{\tr{\alpha}{\l}}  \subseteq \R{\l}$ splits, then $\alpha \leq \fpt{f}$.
%\end{Corollary}

%\begin{proof}  By Lemma \ref{OneConditionSplittingLemma}, $R \cdot f^{\tr{\alpha}{e \l}} \subseteq \R{e \l}$ splits for all $e \geq 1$.  By Key Lemma \ref{FTTruncationLemma}, we have that $\tr{\alpha}{e \l} \leq \tr{\fpt{f}}{e \l}$ for every $e \geq 1$, and taking the limit as $e \to \infty$ gives the desired inequality. 
%\end{proof}

%It is often the case that there exists ``natural'' upper bounds for $F$-pure thresholds.  Indeed,  the notion of $F$-purity can be extended to pairs $\pair{R}{\a}{\lambda}$, where $\a \subseteq R$ is an arbitrary (not necessarily principal) ideal, and  consequently, we may define $\fpt{\a}$ as in Definition \ref{fpt: D2} \cite{TW2004} .  If $f$ is a polynomial over a field $K$ of characteristic $p>0$ with $[K:K^p] < \infty$, one may show that $\fpt{f} \leq \fpt{\a}$, where $\a$ is the monomial ideal associated to $f$ \cite{Polynomials}.  Corollary \ref{OneConditionSplittingCorollary} implies that if $(p^{\l} -1) \cdot \fpt{\a} \in \mathbb{N}$, then $\fpt{f} = \fpt{\a}$ if and only if $R \cdot f^{ \tr{\fpt{\a}}{\l}} \subseteq \R{\l}$ splits.  This is used in \cite{Polynomials} to study the connection between $F$-purity and \emph{log canonical singularities}.  Applying this reasoning to the inequality $\fpt{f} \leq 1$, we arrive at the following generalization of \cite[Proposition 2.16]{MTW2005}.

\begin{Corollary}
$\fpt{f} = 1$ if and only if $R \cdot f^{(p-1)/p} \subseteq \R{}$ splits as a map of $R$-modules.  In particular, if  $R$ is $F$-finite, regular and local, then $R/(f)$ is $F$-pure if and only if $\fpt{f} =1$.
 \end{Corollary}

\begin{proof}
If $\fpt{f} = 1$, then $\tr{\fpt{f}}{1} = \frac{p-1}{p}$, and so $R \cdot f^{(p-1)/p} \subseteq \R{}$ splits over $R$ by Key Lemma \ref{FTTruncationLemma}.  On the other hand, if $R \cdot f^{(p-1)/p} \subseteq \R{}$ splits over $R$, then it follows from setting $\alpha = d = 1$ in Corollary \ref{OneConditionSplittingCorollary} shows that $\fpt{f} \geq 1$, while $\fpt{f} \leq 1$ by Lemma \ref{BasicProperties: L}.

For the second statement, let $\m$ denote the maximal ideal of $R$.  That $R/(f)$ is $F$-pure if and only if $f^{p-1} \notin \bracket{\m}{}$ is known as Fedder's Criteria \cite[Proposition 2.1]{Fed1983}.   On the other hand, $f^{p-1} \notin \bracket{\m}{}$ if and only if $f^{(p-1)/p} \notin \m \cdot \R{}$.  Since $R$ is regular and $F$-finite, $\R{}$ is a finitely generated free $R$-module \cite{Kunz}, and so Nakayama's lemma shows that $f^{(p-1)/p} \notin \m \cdot \R{}$ if and only if $R \cdot f^{(p-1)/p} \subseteq \R{}$ splits over $R$.
\end{proof}

\section{The main results}

In this section, we prove our main results, Theorem \ref{MainTheorem} and Proposition \ref{fptInterval: P}. 

\subsection{More remarks on base $p$ expansions}

We begin with a few straightforward lemmas regarding non-terminating base $p$ expansions that will be used later to impose conditions on the set of all $F$-pure thresholds. 

\begin{Definition}  
\label{Repeat: D}
If $\alpha \in [0,1]$, let $\repeat{\alpha}{e} = \sum \limits_{d \geq 0} \frac{\tr{\alpha}{e}}{p^{ed}} = \tr{\alpha}{e} \cdot \sum \limits_{d \geq 0} \frac{1}{p^{ed}} = \tr{\alpha}{e} \cdot \frac{p^e}{p^e-1}$.
\end{Definition}

\begin{Remark} 
\label{RepeatRemark}
Note that  $\repeat{\alpha}{e}$ is the rational number whose non-terminating base $p$ expansion is  obtained by ``repeating'' the first $e$ digits of the non-terminating base $p$ expansion of $\alpha$.
\end{Remark}

\begin{Lemma}  
\label{RepeatLemma}  If $\alpha \in [0,1]$, then the following hold: 
\begin{enumerate} 
%\item If $\alpha = . \ a_1 \ a_2 \cdots \ \ (\base p)$, then $\repeat{\alpha}{e} = . \ \overline{ a_1 \ a_2 \ \cdots \ a_e} \ (\base p)$.  
\item $\tr{\repeat{\alpha}{e}}{e} = \tr{\alpha}{e}$. \vspace{.05in}
\item $p^e \tr{\alpha}{e} =  (p^e-1) \repeat{\alpha}{e}$.
\end{enumerate}
\end{Lemma}

\begin{proof} These assertions are intuitively clear given the description in Remark \ref{RepeatRemark}, and the task of verifying the details is left to the reader.
\end{proof}

\begin{Lemma}
\label{EquivalentInequalities: L}  Let $\alpha \in [0,1]$.  For $e \geq 1$, the following conditions are equivalent:
\begin{enumerate}
\item \label{EI1} $\down{(p^e-1) \alpha} = p^e \tr{\alpha}{e}$.
\item \label{EI2} $\alpha \leq p^e \tail{\alpha}{e}$.
\item \label{EI3} $\alpha \geq \repeat{\alpha}{e}$.
\end{enumerate}
\end{Lemma}

Lemma \ref{EquivalentInequalities: L} is closely related to Lemma \ref{RoundingLemma}, and its proof proceeds along the same lines.

\begin{proof}
We will assume that $\alpha > 0$.  By definition, $\alpha = \tr{\alpha}{e} + \tail{\alpha}{e}$, and it follows that \begin{equation} \label{star} (p^e-1) \alpha = p^e \tr{\alpha}{e} + p^e \tail{\alpha}{e} - \alpha. \end{equation}  Lemma \ref{TruncationTailLemma} shows that both $\alpha$ and $p^e \tail{\alpha}{e}$ are contained in $(0,1]$, so that $| p^e \tail{\alpha}{e} - \alpha | < 1$.  Thus, \eqref{star} shows that  $\eqref{EI1}$ holds if and only if $\eqref{EI2}$ holds.  By Lemma \ref{RepeatLemma},  $p^e \tr{\alpha}{e} = (p^e-1) \repeat{\alpha}{e}$, and substituting this into $(\ref{star})$ shows that \begin{equation} \label{sides: e} (p^e-1)\left(\alpha - \repeat{\alpha}{e} \right) = p^e \tail{\alpha}{e} - \alpha. \end{equation} From \rref{sides: e}, we see that $\eqref{EI2}$ holds if and only if $\eqref{EI3}$ holds.
\end{proof}

\subsection{The set of all $F$-pure thresholds}

\begin{Definition}
\label{fptsetDefinition}
Let $\fptset{p}$ denote the set of all characteristic $p>0$ $F$-pure thresholds: $ \fptset{p} : = \set{ \fpt{R,f} : f \text{ is a non-zero, non-unit in an $F$-pure ring $R$ of characteristic } p }$.  
\end{Definition}

We stress that both the ambient ring $R$ and the element $f \in R$ are allowed to vary in Definition \ref{fptsetDefinition}. By Lemma \ref{BasicProperties: L}, we have that $\fptset{p} \subseteq [0,1]$.

\begin{Proposition}
\label{Repeating: P}
For any $\lambda \in \fptset{p}$ and $e \geq 1$, we have that $\repeat{\lambda}{e} \leq \lambda$.
\begin{proof}
There exists an $F$-pure ring $R$ and an element $f \in R$ such that $\lambda=\fpt{R,f}$.  Set $\alpha := \repeat{\lambda}{e}$.  By Lemma \ref{RepeatLemma}, we have that 
\begin{enumerate} 
\item \label{pointone} $\tr{\alpha}{e} = \tr{\repeat{\lambda}{e}}{e} = \tr{\lambda}{e}$, and 
\item \label{pointtwo} $(p^{e}-1) \cdot \alpha = p^{e} \tr{\lambda}{e} \in \mathbb{N}.$ 
\end{enumerate}
By Key Lemma \ref{FTTruncationLemma} and $\rref{pointone}$, the inclusion $R \cdot f^{\tr{\alpha}{e} }= R \cdot f^{\tr{\lambda}{e}} \subseteq \R{e}$ splits as a map of $R$-modules.  Corollary \ref{OneConditionSplittingCorollary} and $\rref{pointtwo}$ then imply that $\lambda \geq \alpha$.
\end{proof}
\end{Proposition}

\begin{Corollary}
\label{FTRoundingCorollary}
For any $e \geq 1$ and $ \lambda \in \fptset{p}$, the following hold and are equivalent:
\begin{enumerate}
\item $\down{ (p^e-1) \lambda } = p^e \tr{\lambda}e$.
\item $\lambda \leq p^e \tail{\lambda}e$.
\item $\lambda \geq \repeat{\lambda}{e}$.
\end{enumerate}
\end{Corollary}

\begin{proof}  The third point is Proposition \ref{Repeating: P}, and all points are equivalent by Lemma \ref{EquivalentInequalities: L}.
\end{proof}

Corollary \ref{FTRoundingCorollary} places severe restrictions on the set $\fptset{p}$, as we see in Proposition \ref{fptInterval: P}.

\begin{Proposition} 
\label{fptInterval: P}
For every $e \geq 1$ and  $\beta \in [0,1] \cap \frac{1}{p^e} \cdot \mathbb{N}$, we have that \[ \fptset{p}  \cap \( \hspace{.05in} \beta, \frac{p^e}{p^e-1} \cdot \beta \hspace{.05in} \) = \emptyset. \]
\end{Proposition}

\noindent Proposition \ref{fptInterval: P} generalizes \cite[Proposition 4.3]{BMS2009}, in which $R$ is assumed to be an $F$-finite regular ring.

\begin{proof}
Let $\lambda \in \fptset{p}$. If $\lambda > \beta$, then $\tr{\lambda}{e} \geq \beta$ by Lemma \ref{RoundingLemma}.  Combining this with Proposition \ref{Repeating: P} and Definition \ref{Repeat: D}, we conclude that  $\lambda \geq \repeat{\lambda}{e} = \frac{p^e}{p^e-1} \cdot \tr{\lambda}{e} \geq \frac{p^e}{p^e-1} \cdot \beta$.
\end{proof}

\subsection{On purity at the threshold}  

We conclude by addressing the limiting behavior of the various types of purity given in Definition \ref{PurityDefinition}.

\begin{thm}
\label{MainTheorem}The pair $\pair{R}{f}{\fpt{f}}$ is $F$-pure and is not strongly $F$-pure.  Moreover, $\pair{R}{f}{\fpt{f}}$ is sharply $F$-pure if and only if $(p^e-1) \cdot \fpt{f} \in \mathbb{N}$ for some $e \geq 1$.
\end{thm}

\noindent Theorem \ref{MainTheorem} generalizes \cite[Proposition 2.6]{Hara2006} and \cite[Corollary 5.4 $+$ Remark 5.5]{Schwede2008}, in which $R$ is assumed be an $F$-finite (complete) regular local ring.

\begin{proof}  By Corollary \ref{FTRoundingCorollary}, we have that $\down{ (p^e-1) \fpt{f}} = p^e \tr{\fpt{f}}{e}$, and so Key Lemma \ref{FTTruncationLemma} implies that the inclusion $R \cdot f^{\down{(p^e-1) \fpt{f}}/p^e} = R \cdot f^{\tr{\fpt{f}}{e}} \subseteq \R{e}$ splits over $R$.  This shows that $\pair{R}{f}{\fpt{f}}$ is $F$-pure.

By Lemma \ref{RoundingLemma}, $\up{p^e \fpt{f} }  = p^e \tr{\fpt{f}}{e} + 1$, and applying Key Lemma \ref{FTTruncationLemma} shows that the inclusion $R \cdot f^{\tr{\fpt{f}}{e} + 1/p^e} \subseteq \R{e}$ \emph{never} splits as a map of $R$-modules.  This shows that $\pair{R}{f}{\fpt{f}}$ is not strongly $F$-pure.

Finally, suppose that $\pair{R}{f}{\fpt{f}}$ is sharply $F$-pure.  By Definition \ref{PurityDefinition},  \begin{equation} \label{sfp1: e} R \cdot f^{\up{(p^e-1)\fpt{f}}/p^e} \subseteq \R{e} \text{ splits as a map of $R$-modules for some } e \geq 1.\end{equation}

\noindent By Key Lemma \ref{FTTruncationLemma}, we know that \rref{sfp1: e} holds if and only if   \begin{equation} \label{sfp2: e} \up{(p^e-1)\fpt{f}} \leq p^e \tr{\fpt{f}}{e} = \down{(p^e-1) \fpt{f}}, \end{equation}  where the last equality in \eqref{sfp2: e} comes  from Corollary \ref{FTRoundingCorollary}.  Thus, we must have equality throughout in \eqref{sfp2: e}, which we observe holds if and only if $(p^e-1) \cdot \fpt{f} \in \mathbb{N}$.
\end{proof}

\begin{Example}
\label{CounterExample: E}
There are many instances for which $F$-purity and sharp $F$-purity are distinct.  The simplest example such that $(p^e - 1) \cdot \fpt{f} \notin \mathbb{N}$ for any $e \geq 1$ is the following:  $\fpt{\mathbb{F}_p[[x]], x^p} = \frac{1}{p}$.  For more examples where $\fpt{f}$ is a rational number whose denominator is divisible by $p$, see \cite{Diagonals, Binomials}.
\end{Example}

%\begin{Remark}
%Under the assumption that $R$ is a regular ring such that $\R{}$ is a finite $R$-module, Proposition \ref{fptInterval: P} was originally shown in \cite[Proposition 4.3]{BMS2009} via the theory of \emph{$F$-jumping numbers}.
%\end{Remark}

%-----------------------------
\bibliographystyle{alpha}
\bibliography{refs}
%-----------------------------

\end{document}